\documentclass[10pt,a4paper]{article}
\usepackage{amsmath,amssymb,amsthm,mathtools,multirow,graphicx,ifpdf}
\usepackage[mathscr]{eucal}
\usepackage[subnum]{cases}
\usepackage{fixltx2e}
\newtheorem{theorem}{Theorem}[section]

\newtheorem{corollary}[theorem]{Corollary}

\theoremstyle{definition}
\newtheorem{definition}[theorem]{Definition}
\newtheorem{remark}[theorem]{Remark}
\newtheorem{example}[theorem]{Example}

\numberwithin{equation}{section}

\usepackage{enumitem}

\usepackage{authblk}

\usepackage{bbm}

\allowdisplaybreaks

\title{Asynchronous discrete dynamical systems}
\author{
Stefan Siegmund\\
Faculty of Mathematics, 
TU Dresden,
Germany\\
Email: stefan.siegmund@tu-dresden.de
\and
Petr Stehl\'\i k\\
Faculty of Applied Sciences,
University of West Bohemia,
Czech Republic\\
Email: pstehlik@kma.zcu.cz
}

\newcommand{\ty}{t^{*_\nu}}
\newcommand{\tx}{t^{*_\mu}}

\newcommand{\C}{\mathbb{C}}
\newcommand{\R}{\mathbb{R}}

\newcommand{\N}{\mathbb{N}}
\newcommand{\TS}{\mathbb{T}}
\newcommand{\T}{\mathbb{T}}
\newcommand{\Tx}{\TS_\mu}
\newcommand{\Ty}{\TS_\nu}
\newcommand{\TT}{\TS_T}
\newcommand{\Z}{\mathbb{Z}}

\DeclareMathOperator{\indicator}{\mathbbm{1}}

\usepackage{array}
\usepackage{tabularx}

\begin{document}
\maketitle

\begin{abstract}
We study two coupled discrete-time equations with different (asynchronous) periodic time scales. The coupling is of the type sample and hold, i.e., the state of each equation is sampled at its update times and held until it is read as an input at the next update time for the other equation. We construct an interpolating two-dimensional complex-valued system on the union of the two time scales and an extrapolating four-dimensional system on the intersection of the two time scales. We discuss stability by several results, examples and counterexamples in various frameworks to show that the asynchronicity can have a significant impact on the dynamical properties.
\end{abstract}


\section{Introduction}

The notion of asynchronous control system \cite{Hassibi:Boyd:How1999} often denotes models for asynchronously occurring discrete events which trigger a continuous time system, e.g., control systems in which signals are transmitted over an asynchronous network. In this paper we coin the notion of asynchronous discrete dynamical system to denote models with two inherently different discrete time scales. Little is known about about linear discrete-time systems consisting of coupled components each of which has its own potentially different time scale.  This is despite the fact that asynchronous discrete-time phenomena are observed in the real world, both in the human society and in the animal kingdom. A number of disciplines within the social and natural sciences have attempted to formally explore these phenomena.

For example, in economics, asynchronous time scales arise naturally since \cite{T} \emph{``decisions by economic agents are reconsidered daily or hourly, while others are reviewed at intervals of a year or longer''}. These intervals are driven by various costs and benefits that individuals, companies and governments face if they want to reconsider their decisions \cite{K}. Therefore there exist doubts about the prevalence of a synchronous approach and it has been argued \cite{LM} that \emph{``\ldots the synchronized move is not an unreasonable model of repetition in certain settings, but it is not clear why it should necessarily be the benchmark setting.''}

Similarly, in biology, \cite{Murray2} examined, e.g., wolf activities occurring over yearly, seasonal and daily time scales, which significantly affects the modeling of wolf-deer interactions. Interestingly, periodically varying insect populations with distinct life cycle periods coexist in many regions around the world. Their periods range from the most commonly observed one year, through many species with periods of two or three years to periods of 13 and 17 years of periodic cicadas of genus \emph{Magicada}, see the survey paper \cite{HVS}. For other examples of coupled systems and models with different time scales in social and natural sciences, see, e.g., \cite{LM, LS1, LS2, SBCS}.

From the mathematical point of view, there has been a considerable interest in the role of timing structures both from the numerical as well as analytical point of view. A well-developed mathematical theory on dynamic equations on a single time-scale can be found, e.g., in \cite{BP, H}. Note that even the elementary notion of stability depends strongly on the underlying time scale \cite{PSW}. This effect is apparent when we explore asynchronous linear systems below, focusing on two periodic time scales with different periods in particular. 

There exist sporadic continuous-time approaches in which distinct periodic physical processes (e.g., mechanical, thermal, diffusion, chemical) are considered. The temporal homogenization technique \cite{AP,YF} considers such continuous phenomena in the special case when there co-exist fast and slow oscillatory processes with significantly different periods. The large ratio of the two periods allows to split the analysis in local and global problems. Our approach not only considers discrete time instead but is more general in the sense that arbitrary (not necessarily significantly different) periods are taken into account.

The paper is organized as follows. In $\S$\ref{sec:pf} we formulate two-dimensional asynchronous dynamical systems and introduce the necessary notation. In $\S$\ref{sec:linear:system} we associate an extended four-dimensional system to the asynchronous linear one and study the solution operator of the system. In $\S$\ref{sec:timeone} we analyze an interpolated dynamical system on a finer time scale. Consequently, we provide results and examples for various special cases in $\S$\ref{sec:synchronous}-$\S$\ref{sec:commensurable}. We conclude by formulating open questions and identifying directions for further research in $\S$\ref{sec:final:remarks}.


\section{Problem formulation -- 2 equations}\label{sec:pf}
For $\rho>0$, we define the periodic (or regular) one-sided discrete time scale
\begin{equation*}
   \TS_\rho 
   \coloneqq 
   \left\lbrace 0, \rho, 2\rho, \ldots \right \rbrace,
\end{equation*}
and the (forward) difference operator $\Delta_\rho \colon \mathbb{R}^{\TS_\rho} \to \mathbb{R}^{\TS_\rho}$ by
\begin{equation*}
   \Delta_\rho x(t) 
   = 
   \frac{x(t+\rho)-x(t)}{\rho}
   \qquad
   (t \in \TS_\rho),
\end{equation*}
for $x \colon \TS_\rho \to \mathbb{R}$.
The lag operator  $\cdot^{*\rho} \colon \mathbb{R}_{\geq 0} \to \mathbb{T}_\rho$ on $\mathbb{R}_{\geq 0} \coloneqq \{x \in \mathbb{R} \,|\, x \geq 0\}$ is defined by
\begin{equation*}
   t \mapsto t^{*\rho} 
   \coloneqq 
   \rho\left\lfloor \frac{t}{\rho} \right\rfloor 
   = 
   \max \left\lbrace s\in\mathbb{T}_\rho \colon s\leq t \right\rbrace
   \qquad
   (t \in \mathbb{R}_{\geq 0}),
\end{equation*}
and satisfies $t-t^{*\rho} \in [0,\rho)$ for $t \in \mathbb{R}_{\geq 0}$. 

Let $\nu > 0$. Using the lag operator $\cdot^{*\nu}$, a function $z_{\mathbb{T}_\nu} \colon \mathbb{T}_\nu \to \mathbb{R}$ on $\mathbb{T}_\nu$ can naturally be extended to a function $z_{\mathbb{R}_{\geq 0}} \colon \mathbb{R}_{\geq 0} \to \mathbb{R}$ on $\mathbb{R}_{\geq 0}$ by defining
\begin{equation*}
   z_{\mathbb{R}_{\geq 0}}(t)
   \coloneqq
   z_{\mathbb{T}_\nu}(t^{*\nu})
   \qquad
   (t \in \mathbb{R}_{\geq 0}).
\end{equation*}
The extended function $z_{\mathbb{R}_{\geq 0}}$ realizes the principle of \emph{sample and hold} in systems theory \cite[Section 1.4]{Ogata1995} and can be used in the analysis of discrete-time equations \cite{Slavik2012}. For every $\tau \in \mathbb{T}_\nu$ the value $z_{\mathbb{T}_\nu}(\tau)$ is sampled and held constant for one period $\nu$, i.e.,
\begin{equation*}
   z_{\mathbb{R}_{\geq 0}}(t)
   =
   z_{\mathbb{T}_\nu}(\tau)
   \qquad
   (t \in [\tau, \tau+ \nu)).
\end{equation*}
Let $\mu > 0$. The restriction $z_{\mathbb{T}_\mu} \coloneqq z_{\mathbb{R}_{\geq 0}}|_{\mathbb{T}_\mu}$ of $z_{\mathbb{R}_{\geq 0}}$ to $\mathbb{T}_\mu$, satisfies
\begin{equation*}
   z_{\mathbb{T}_\mu}(t)
   =
   z_{\mathbb{T}_\nu}(t^{*\nu})
   \qquad
   (t \in \mathbb{T}_\mu),
\end{equation*}
and reads the value of $z_{\mathbb{T}_\nu}$ which was sampled at $t^{*\nu}$ and held until $t \in \mathbb{T}_\mu$. 

We define an asynchronous discrete dynamical system in the following way. Consider periods $\mu, \nu > 0$ and the matrix
\begin{equation*}
  P 
  = 
  \begin{pmatrix}
    \alpha & \beta \\
    \gamma & \delta 
  \end{pmatrix}
  \in \mathbb{R}^{2\times 2}.
\end{equation*}
Then the two coupled equations on the time scales $\Tx$ and $\Ty$
\begin{equation}\label{e:problem:2}
   \begin{cases}
      \Delta_\mu x(t) = \alpha x(t) + \beta y(\ty), & t\in\Tx, \\
      \Delta_\nu y(t) = \gamma x(\tx) + \delta y(t), & t\in\Ty,
   \end{cases}
\end{equation}
are called $(\mu,\nu)$-asynchronous difference equation (or $(\mu,\nu)$-asynchronous discrete time dynamical system) with parameters $P$. See Figure \ref{f:53} for an illustration of $\Tx$ and $\Ty$.
\begin{figure}
\begin{center}
\includegraphics[width=12cm]{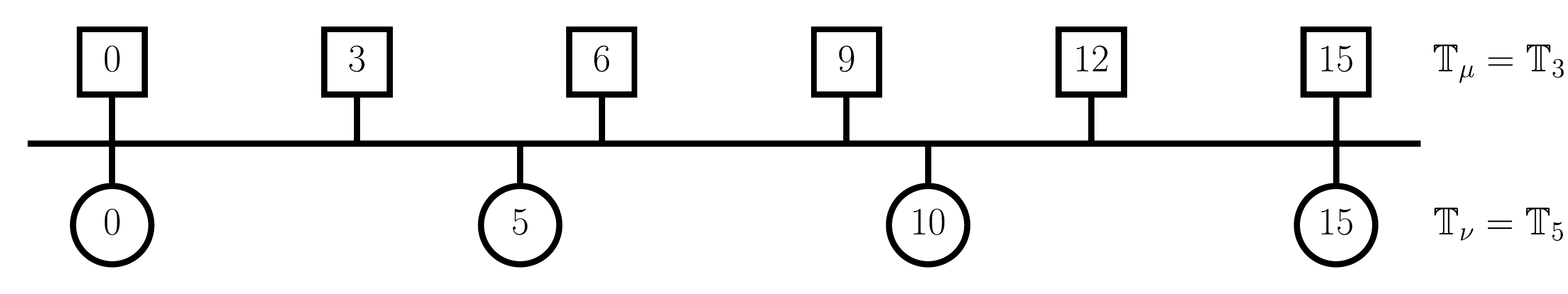}
\vspace*{-2ex}
\end{center}
\caption{Time scales $\TS_3$ and $\TS_5$ of a $(3,5)$-asynchronous discrete dynamical system \eqref{e:problem:2}.}\label{f:53}
\end{figure}

A tuple $(x,y)$ of sequences $x \colon \Tx \rightarrow \mathbb{R}$, $y \colon \Ty \rightarrow \mathbb{R}$, is called \emph{solution of} \eqref{e:problem:2}, if it satisfies \eqref{e:problem:2}.

Throughout the paper we use the notation $\sigma(t)$, $\sigma:\TS\rightarrow\TS$, to denote the successor of an element $t$ of a given time scale $\TS$. This operator is commonly known as the forward-jump operator in the theory of time scales \cite{BP}. Naturally, its value depends strongly on the considered time scale, e.g., $\sigma(0)=\rho$ for $\TS=\TS_\rho$.

\begin{example}\label{x:VAR}
In economics, dynamics of variables naturally includes data of different frequencies.  A typical problem arises from different natural time scales. Let us consider a simplistic model of fiscal spending involving gross domestic product (GDP) and government spending. GDP data are usually observed quarterly and would naturally imply a time scale $\mathbb{T}_{3}$ with $\mu=3$ months. On the other hand, the government spending is determined by a time series with yearly frequency as government budget is approved once a year. In this case the natural time scale $\mathbb{T}_{12}$ has a frequency of $\nu=12$ months.

Denoting GDP by $x$ and government spending by $y$, a synchronized simplistic fiscal model has the form:
\begin{equation}\label{eq:synchronous:fiscal}
\begin{cases}
      \Delta x(t) = \alpha x(t) + \beta y(t), & t\in\T_{12}, \\
      \Delta y(t) = \gamma x(t) + \delta y(t), & t\in\T_{12},
   \end{cases}
\end{equation}

Alternatively, we could consider the asynchronous model \eqref{e:problem:2} and allow GDP $x$ and government spending $y$ to follow their own time scales - quarterly $\mathbb{T}_{3}$ for $x$ and annual $\mathbb{T}_{12}$ for $y$.

Apparently, the asynchronous model is closer to how economic agents observe macroeconomic data. In this paper we are mainly interested in qualitative differences between the synchronous model \eqref{eq:synchronous:fiscal} and the asynchronous model \eqref{e:problem:2}. Naturally, from the point of view of macroeconomic applications there are numerous other questions, which we do not discuss here. Are the real data time series better explained by the asynchronous model \eqref{e:problem:2}? Can the asynchronous formulation \eqref{e:problem:2} provide better forecasts of the future?

\end{example}

\section{Linear system representation}\label{sec:linear:system}

It is well-known that for difference equations a delay can be eliminated by increasing the dimension of the system (see, e.g., \cite{EZ}). We associate a 4-dimensional linear system to \eqref{e:problem:2} by storing appropriate delayed values of $x(t), y(t)$ in auxiliary variables $\underline{x}(t), \underline{y}(t)$. For this to work, we need to incorporate all times on which dynamics happens in \eqref{e:problem:2}, i.e., we consider the union 
\[
  \mathbb{T} \coloneqq \Tx \cup \Ty
\]
of the two time scales of \eqref{e:problem:2}. 

\begin{theorem}[Linear system representation]\label{thm:representation}
Let $\mu,\nu>0$ and $\mathbb{T} = \Tx \cup \Ty$. Then there exists a unique $A:\TS\rightarrow \mathbb{R}^{4\times 4}$ and the corresponding four dimensional dynamic equation
\begin{equation}\label{e:linear:system}
\begin{pmatrix}
u \\
\underline{u} \\
v \\
\underline{v}
\end{pmatrix} (\sigma(t)) =
 A(t) \begin{pmatrix}
u \\
\underline{u} \\
v \\
\underline{v}
\end{pmatrix} (t), \qquad t\in\TS, \begin{pmatrix}
u \\
\underline{u} \\
v \\
\underline{v}
\end{pmatrix} \in\mathbb{R}^4,
\end{equation}
such that for arbitrary sequences $x:\Tx\rightarrow\mathbb{R}$ and $y:\Ty\rightarrow\mathbb{R}$ the following statements are equivalent:
\begin{enumerate}
  \item[(i)] $(x,y)$ is a solution of \eqref{e:problem:2},

  \item[(ii)] the solution $(u,\underline{u},v,\underline{v}) \colon \T \rightarrow \mathbb{R}^4$  of \eqref{e:linear:system} with $u(0)=x(0)$, $\underline{u}(0)=0$, $v(0)=y(0)$, $\underline{v}(0)=0$ satisfies $u|_{\Tx}= x$ and $v|_{\Ty}= y$.  
\end{enumerate}
\end{theorem}

\begin{proof}
Consider the following difference equation on $\TS$:
\begin{numcases}{\label{(a)}\hspace*{-1.6cm}u(\sigma(t))=}
   u(t) \quad \text{if } \sigma(t) \notin \Tx, & \label{(a1)}
   \\[1ex]
   (1 + \alpha \mu)u(t) + \beta \mu v(t) \quad \text{if } \sigma(t) \in \Tx \wedge t \in \Tx\cap\Ty, & \label{(a2)}
   \\[1ex]
   (1 + \alpha \mu)u(t) + \beta \mu \underline{v}(t) \quad \text{if } \sigma(t) \in \Tx \wedge t \notin \Tx\cap\Ty,& \label{(a3)}
\end{numcases}
\addtocounter{equation}{-1}
\begin{numcases}{\label{(b)}\underline{u}(\sigma(t))=}
   \setcounter{equation}{4}
   u(t) \quad \text{if } t\in\Tx\cap\Ty \vee (t\notin \Ty \wedge \sigma(t)\in \Ty),\phantom{aaaaaaaaaaaaaa} & \label{(b1)}
   \\[1ex]
   \underline{u}(t) \quad \text{otherwise,} & \label{(b2)}
\end{numcases}
\addtocounter{equation}{-1}
\begin{numcases}{\label{(c)}\hspace*{-1.6cm}v(\sigma(t))=}
   \setcounter{equation}{6}
   v(t) \quad \text{if } \sigma(t) \notin \Ty, & \label{(c1)}
   \\[1ex]
   (1 + \delta \nu)v(t) + \gamma \nu u(t)\quad \text{if } \sigma(t) \in \Ty \wedge t \in \Tx\cap\Ty, & \label{(c2)}

   \\[1ex]
   (1 + \delta \nu)v(t) + \gamma \nu \underline{u}(t) \quad \text{if } \sigma(t) \in \Ty \wedge t \notin \Tx \label{(c3)}
\end{numcases}
\addtocounter{equation}{-1}
\begin{numcases}{\label{(d)}\underline{v}(\sigma(t))=}
   \setcounter{equation}{9}
   v(t)  \quad \text{if } t\in\Tx\cap\Ty \vee (t\notin \Tx \wedge \sigma(t)\in \Tx),\phantom{aaaaaaaaaaaaaa} & \label{(d1)}
    \\[1ex]
   \underline{v}(t) \quad \text{otherwise.} \label{(d2)}
\end{numcases}
It is of the form \eqref{e:linear:system}. We first prove the following statement:
\begin{equation}\label{e:statement}
   (u,\underline{u},v,\underline{v}) \text{ solves } 
   \eqref{(a)}   
   \quad \Rightarrow \quad 
   (u|_{\Tx}, v|_{\Ty}) \text{ solves }  \eqref{e:problem:2}.
\end{equation}
To this end, let $(u,\underline{u},v,\underline{v}) \colon \T \rightarrow \mathbb{R}^4$  be a solution of \eqref{(a)}. In order to prove \eqref{e:statement}, we show that 
\begin{alignat}{2} \label{e:proof:4:u}
   u(t+\mu) &= (1+\alpha \mu) u(t) + \beta \mu v(\ty) &&\qquad (t \in \Tx),
\\ \label{e:proof:4:v}
   v(t+\nu) &= \gamma \nu u(\tx) + (1+\delta \nu) v(t)  &&\qquad  (t \in \Ty).
\end{alignat}
We only show \eqref{e:proof:4:u}, the latter relation \eqref{e:proof:4:v} is proved analogously. To show \eqref{e:proof:4:u}, assume that $t \in \T_\mu$, then $t + \mu \in \T_\mu$, too. Define
\begin{equation*}
   r \coloneqq \max \{s \in \T \colon s < t + \mu\}.
\end{equation*}
Obviously, $t \leq r < t + \mu$. We distinguish between the following cases.

Case 1: If $r \in \T_\mu \cap \T_\nu$, then clearly $r = t = t^{*\nu}$ and \eqref{e:proof:4:u} follows from (\ref{(a)}b).

Case 2: If $r \not\in \T_\mu \cap \T_\nu$, then we can either have  $r \in \T_\mu \setminus \T_\nu$ or  $r \in \T_\nu \setminus \T_\mu$:

Case 2.1: If $r \in \T_\mu \setminus \T_\nu$, then $r=t$ and (\ref{(d)}i)-(\ref{(d)}j) imply that
\begin{equation}\label{e:case2.1}
   \underline{v}(r) = \underline{v}(t) = v(t^{*\nu}).
\end{equation}

Case 2.2: If $r \in \T_\nu \setminus \T_\mu$, then (\ref{(d)}j) yields
\begin{equation*}
   \underline{v}(\sigma(t)) = \underline{v}(s) = \underline{v}(r),
\end{equation*}
for all $s\in\T_\nu$ with $\sigma(t)<s<r$. Consequently, (\ref{(d)}i) implies that
\begin{equation}\label{e:case2.2}
   \underline{v}(r) = \underline{v}(\sigma(t)) = v(t^{*\nu}).
\end{equation}
By assumptions of Case 2 and (\ref{(a)}a), $u(r) = u(t)$. Relations \eqref{e:case2.1} and \eqref{e:case2.2} yield that $\underline{v}(r) = v(t^{*\nu})$. Hence (\ref{(a)}c) implies that
\begin{align*}
   u(t + \mu)
   &=
   (1 + \alpha \mu) u(r) + \beta \mu \underline{v}(r)
\\
   &=
   (1 + \alpha \mu) u(t) + \beta \mu v(t^{*\nu}),
\end{align*}
i.e., \eqref{e:proof:4:u} holds in this case as well.

$(i)\Rightarrow (ii)$.
Let $(x,y)$ be a solution of \eqref{e:problem:2} and let $(u,\underline{u},v,\underline{v}) \colon \T \rightarrow \mathbb{R}^4$ be the solution of \eqref{(a)} with $u(0)=x(0)$, $\underline{u}(0)=0$, $v(0)=y(0)$, $\underline{v}(0)=0$. Then by \eqref{e:statement}, $(u|_{\Tx}, v|_{\Ty})$ solves \eqref{e:problem:2}. Since $u(0)=x(0)$, $v(0)=y(0)$, it follows that $u|_{\Tx}= x$ and $v|_{\Ty}= y$.

$(ii)\Rightarrow (i)$. Let $x:\Tx\rightarrow\mathbb{R}$ and $y:\Ty\rightarrow\mathbb{R}$ be given and let $(u,\underline{u},v,\underline{v}) \colon \T \rightarrow \mathbb{R}^4$ be the solution of \eqref{(a)} with $u(0)=x(0)$, $\underline{u}(0)=0$, $v(0)=y(0)$, $\underline{v}(0)=0$. Then by (ii), $u|_{\Tx}= x$ and $v|_{\Ty}= y$, and \eqref{e:statement} implies that $(x,y)$ solves \eqref{e:problem:2}.
\end{proof}

Using the indicator function $\indicator_M \colon \R \to \{0,1\}$ for a set $M \subseteq \R$,
\begin{equation*}
   \indicator_{M}(t)
   \coloneqq
   \begin{cases}
      1 & \text{if } t \in M,
   \\
      0 & \text{if } t \not\in M,
   \end{cases}
\end{equation*}
we get the following explicit representation of the coefficient matrix $A(t)$ of the linear system representation \eqref{e:linear:system}.

\begin{corollary}[Explicit form of linear system representation]\label{c:matrices}
Under the assumptions of Theorem \ref{thm:representation} the coefficients $A \colon \T \rightarrow \R^{4 \times 4}$ are given by 
\begin{equation*}
   A(t) 
   \coloneqq
   A_{i j k \ell}
\end{equation*}
with $(i,j,k,\ell) = \big(\indicator_{\T_{\mu}}(t), \indicator_{\T_{\mu}}(\sigma(t)), \indicator_{\T_{\nu}}(t), \indicator_{\T_{\nu}}(\sigma(t))\big)$ and $A_{i j k \ell}$ as defined by one of the $9$ cases listed in Table \ref{t:9cases}.
\end{corollary}

\begin{proof}
For each $t \in \T = \T_\mu \cup \T_\nu$ there are $3$ possible disjoint cases, either
\begin{equation*}
   t \in \T_\mu \cap \T_\nu
   \qquad \text{or} \qquad
   t \in \T_\mu \setminus \T_\nu
   \qquad \text{or} \qquad
   t \in \T_\nu \setminus \T_\mu.
\end{equation*}
Similarly, there are $3$ cases for $\sigma(t) \in \T$. Consequently, there are $9$ possible combinations for the values of the indicator functions in the quadruple
\begin{equation*}
   (i,j,k,\ell)
   \coloneqq
   \big(\indicator_{\T_{\mu}}(t), 
   \indicator_{\T_{\mu}}(\sigma(t)), 
   \indicator_{\T_{\nu}}(t), 
   \indicator_{\T_{\nu}}(\sigma(t))\big)
\end{equation*}
which are listed and illustrated in Table \ref{t:9cases}. For each of those cases, we can use the linear system representation \eqref{(a)} to compute $A_{i j k \ell}$ as listed in Table \ref{t:9cases}.
\end{proof}

Consequently, we are able to introduce a solution operator of \eqref{e:problem:2}.

\begin{corollary}[Solution operator for asynchronous discrete time dynamical system]\label{c:representation:matrix}
Let $\mathbb{T}=\mathbb{T}_\mu \cup \mathbb{T}_\nu$, $(\mathbb{T})^2_{\geq} \coloneqq \{(t,t_0) \in \mathbb{T} \times \mathbb{T} \colon t \geq t_0\}$ and let $\Phi \colon (\mathbb{T})^2_{\geq} \rightarrow \mathbb{R}^{4 \times 4}$, $(t,t_0) \mapsto \Phi(t,t_0)$, denote the evolution operator (2-parameter process) of \eqref{(a)}, i.e., for $t_0 \in \T$
\[
  (t,t_0) \mapsto \Phi(t,t_0) (x_0, \underline{x}_0, y_0, \underline{y}_0)^{\top}
\]
is the solution of the initial value problem \eqref{(a)}, $(x(t_0), \underline{x}(t_0), y(t_0), \underline{y}(t_0)) = (x_0, \underline{x}_0, y_0, \underline{y}_0)$. Define the \emph{solution operator for \eqref{e:problem:2}} as $\Psi \colon (\mathbb{T})^2_{\geq} \rightarrow \mathbb{R}^{2 \times 2}$, $(t,t_0) \mapsto \Psi(t,t_0)$,
\begin{equation*}
  \Psi(t,t_0)
  \coloneqq
  \begin{pmatrix}
     \Phi_{11}(t,t_0) & \Phi_{13}(t,t_0)
  \\
     \Phi_{31}(t,t_0) & \Phi_{33}(t,t_0)
  \end{pmatrix}.
\end{equation*}
Then for $(x_0,y_0) \in \R^2$ and $x, y \colon \T \rightarrow \R$ with
\begin{equation*}
   \begin{pmatrix}
     x(t)
   \\
     y(t)
  \end{pmatrix}
  \coloneqq
   \Psi(t,0)
   \begin{pmatrix}
     x_0
   \\
     y_0
   \end{pmatrix},
\end{equation*}
the tuple $(x_{| \T_\mu}, y_{| \T_\nu})$ of restrictions is a solution of \eqref{e:problem:2}.
\end{corollary}

\begin{proof}
Let $t\in\T$. The set $\T \cap [0,t]$ contains finitely many elements $0 = t_0 < t_1 < \dots < t_k = t$. Theorem~\ref{thm:representation} implies that with the matrices $A(t_i)$ from Corollary \ref{c:representation:matrix} and Table \ref{t:9cases},
\begin{align*}
	(x(t),\underline{x}(t),y(t),\underline{y}(t))^\top 
	&\coloneqq \Phi(t,0) (x_0,0,y_0,0)^\top \\ 
	&= \Phi(t_k,t_{k-1}) \cdots \Phi(t_2,t_1) \Phi(t_1,t_0) (x_0,0,y_0,0)^\top \\
	 &= A(t_{k-1}) \cdots A(t_1)A(t_0) (x_0,0,y_0,0)^\top
\end{align*}
is a solution of \eqref{(a)}. Writing $\Phi = (\Phi_{i j})_{i,j=1, \dots, 4}$, we get for $t \in \T$
\[
\Phi(t,0) =
	\begin{pmatrix}
	\Phi_{11} & \Phi_{12} & \Phi_{13} & \Phi_{14} \\
	\Phi_{21} & \Phi_{22} & \Phi_{23} & \Phi_{24} \\
	\Phi_{31} & \Phi_{32} & \Phi_{33} & \Phi_{34} \\
	\Phi_{41} & \Phi_{42} & \Phi_{43} & \Phi_{44} \\
	\end{pmatrix}(t,0)
\]
and hence
\begin{align*}
\begin{pmatrix}
x(t)\\
y(t)
\end{pmatrix} 
&=
\begin{pmatrix}
\pi_1 \circ \Phi(t,0)\cdot (x_0,0,y_0,0)^\top \\
\pi_3 \circ \Phi(t,0)\cdot (x_0,0,y_0,0)^\top
\end{pmatrix}
\\
&=
\begin{pmatrix}
\Phi_{11} & \Phi_{13} \\
\Phi_{31} & \Phi_{33}
\end{pmatrix} (t,0)\cdot (x_0,y_0)^\top \\
&= \Psi(t,0) \cdot (x_0,y_0)^\top.
\end{align*}
By Theorem \ref{thm:representation}, the tuple $(x_{| \T_\mu}, y_{| \T_\nu})$ solves \eqref{e:problem:2}.
\end{proof}

\begin{remark}
Note that $\Psi(t,t_0)$ is not a $2$-parameter process on $\T$, since for arbitrary $t, \tau, s \in \T$,
\begin{equation}\label{e:semigroup}
   \Psi(t, \tau) \Psi(\tau, s) = \Psi(t, s)
\end{equation}
does not necessarily hold. However, \eqref{e:semigroup} holds for all $t, \tau, s \in \T_\mu \cap \T_\nu$. Moreover, if $\mu$ and $\nu$ are commensurable, i.e., $\T_\mu \cap \T_\nu = \T_T$ with $T = \operatorname{lcm}(\mu, \nu)$, then for $k \in \Z$
\begin{equation*}
   \Psi(T, 0) = \Psi((k + 1) T, k T).
\end{equation*}
Consequently, $ \Psi(kT,0) = \prod_{i=1}^k \Psi(T,0)$ for $k \in \N$ in this case.
\end{remark}

We conclude this section with a simple illustration of Corollary \ref{c:representation:matrix}.

\begin{example}\label{x:12}
Let us consider asynchronous dynamics with $\mu=1$ and $\nu=2$, i.e.,
\begin{align*}
\Tx &= \{0,1,2, \ldots \}, \\
\Ty &= \{0,2,4, \ldots \}, \\
\T &= \Tx\cup\Ty = \{0,1,2, \ldots \}.
\end{align*}
Since $0,2\in\Tx\cap\Ty$ and $1\in\Tx\setminus\Ty$ we get that $\Phi(2,0)=A_{1101}\cdot A_{1110}$ where (see Table \ref{t:9cases}):
\[
A_{1110} = \begin{pmatrix}
 1+\alpha & 0 & \beta  & 0 \\
 1 & 0 & 0 & 0 \\
 0 & 0 & 1 & 0 \\
 0 & 0 & 1 & 0 \\
\end{pmatrix},
\quad
A_{1101} = \begin{pmatrix}
 1+\alpha  & 0 & 0 & \beta  \\
 1 & 0 & 0 & 0 \\
 0 & 2 \gamma  & 1+2 \delta & 0 \\
 0 & 0 & 0 & 1 \\
\end{pmatrix}.
\quad
\]
Consequently,
\[
\Phi(2,0)=
\begin{pmatrix}
 (1+\alpha)^2 & 0 & (1+\alpha) \beta +\beta  & 0 \\
 1+\alpha & 0 & \beta  & 0 \\
 2 \gamma  & 0 & 1+2 \delta & 0 \\
 0 & 0 & 1 & 0
\end{pmatrix},
\quad
\Psi(2,0)=\begin{pmatrix}
 (1+\alpha)^2 & (1+\alpha) \beta +\beta   \\
 2 \gamma  &  1+2 \delta 
\end{pmatrix}.
\]
\end{example}

\section{Interpolated dynamics}\label{sec:timeone}
Recall that for $p \in \N$ a nonsingular matrix $M \in \R^{n \times n}$ has a complex $p$-th root $\sqrt[p]{M} \in \C^{n \times n}$ \cite[Section 7.1, pp.\ 173-174]{Hi}. If $\mu,\nu>0$ are commensurable, i.e., $\T_\mu \cap \T_\nu = \T_T$ with $T > 0$, there exist $k, \ell \in \N$ with $k \mu = \ell \nu = T$. We define the time scale $\T_\tau$ with
\begin{equation*}
   \tau 
   \coloneqq 
   \tfrac{T}{k \ell}
   =
   \tfrac{\mu}{\ell}
   =
   \tfrac{\nu}{k}.
\end{equation*}
$\T_\tau$ refines or interpolates $\T_\mu$ as well as $\T_\nu$, since $\T_\mu \subseteq \T_\tau$ and $\T_\nu \subseteq \T_\tau$. We can construct an associated complex difference equation 
\begin{equation}\label{e:problem:uv}
   \begin{pmatrix}
      u(t+\tau) \\ v(t+\tau)
   \end{pmatrix}
   =
   B
   \begin{pmatrix}
      u(t) \\ v(t)
   \end{pmatrix}
   \qquad
   (t \in \T_\tau),
\end{equation}
which interpolates \eqref{e:problem:2}, by defining
\[
   B 
   \coloneqq \sqrt[k \ell]{\Psi(T,0)} \in \mathbb{C}^{2 \times 2}, 
\]
and
\begin{equation}\label{e:timeone}
   \Psi(t,s) 
   \coloneqq 
   B^{\frac{t-s}{\tau}}
   \qquad
   \text{for } t, s, \in \T_\tau, t \geq s.
\end{equation}
By \eqref{e:timeone}, $B = \Psi(\tau,0)$ and $\Psi(T,0) = B^{k \ell}$.

For every solution $(x,y)$ of \eqref{e:problem:2} and solution $(u,v)$ of \eqref{e:problem:uv} with $x(0) = u(0)$, $y(0) = v(0)$, we have
\[
   \begin{pmatrix}
      x(t) \\ y(t)
   \end{pmatrix}
   =
   \begin{pmatrix}
      u(t) \\ v(t)
   \end{pmatrix}
   \qquad \text{for all } t \in \TT. 
\]
Note that in general $x(t) = u(t)$ does not necessarily hold for $t \in \T_\mu \setminus \TT$, also $y(t) = v(t)$ is not necessarily true for $t \in \T_\nu \setminus \TT$.

We illustrate these notions by going back to Example \ref{x:12}.
\begin{example}\label{x:12:time1map}
Let us consider the asynchronous discrete time dynamical system \eqref{e:problem:2} with $\mu=1$, $\nu=2$ and parameters $\alpha=2$, $\beta=1$, $\gamma=-1$ and $\delta=1$. Following Example \ref{x:12}, we observe that
\[
\Psi(2,0)= \begin{pmatrix}
 9 & 4   \\
 -2  &  3
 \end{pmatrix}.
\]
Since $T=\operatorname{lcm}(1, 2)=2$, we compute the matrix square root of $\Psi(2,0)$ and determine that
\[
   B \coloneqq \Psi(1,0) \coloneqq \sqrt{\Psi(2,0)} = 
\begin{pmatrix}
 -\sqrt{5}+2 \sqrt{7} & -2 \sqrt{5}+2 \sqrt{7} \\
 \sqrt{5}-\sqrt{7} & 2 \sqrt{5}-\sqrt{7}
 \end{pmatrix}  
 \approx  
 \begin{pmatrix}
 3.05543 & 0.819367 \\
 -0.409683 & 1.82638
 \end{pmatrix}.
\]
\end{example}

\section{Synchronous case $\mu=\nu$}\label{sec:synchronous}
Let us consider the standard synchronous setting $\mu = \nu$ as a benchmark case first. We have $\TS=\Tx=\Ty$ and the system \eqref{e:problem:2} can be rewritten as
\begin{equation}\label{e:problem:2:equal}
\begin{cases}
\Delta_\mu x (t) = \alpha x(t) + \beta y(t),\\
\Delta_\mu y (t) = \gamma x(t) + \delta y(t),
\end{cases}
\qquad (t\in\mathbb{T}).
\end{equation}
In order to be able to compare various asynchronous dynamics later, we use special notation for the solution operator $\Psi$ which indicates the periodicities via upper indices
\[
\Psi^{\mu,\mu}(\mu,0) :=\left(\begin{array}{cc}
1+\mu\alpha & \mu\beta \\
\mu\gamma & 1+\mu\delta 
\end{array}
 \right).
\]
We can rewrite system \eqref{e:problem:2:equal} as
\[
\binom{x(t+\mu)}{y(t+\mu)}  =  \left(\begin{array}{cc}
1+\mu\alpha & \mu\beta \\
\mu\gamma & 1+\mu\delta 
\end{array}
 \right)\binom{x(t)}{y(t)} = \Psi^{\mu,\mu}(\mu,0) \binom{x(t)}{y(t)}.
\]
Denoting by $\lambda(M)$ the set of all eigenvalues of a matrix $M$, we can claim the following result, an alternative of a well-known result from the theory of discrete-dynamical systems, see \cite{KP}.
\begin{theorem}\label{thm:synchronous}
Let $\lambda(\Psi^{\mu,\mu}(\mu,0))\subset B(0,1)$ or, equivalently, $\lambda(P)\subset B\big(-\frac{1}{\mu},\frac{1}{\mu} \big)$. Then $o \coloneqq (0,0) \in \mathbb{R}^2$ is a globally asymptotically stable solution of \eqref{e:problem:2:equal}, i.e., every solution $(x,y)$ of \eqref{e:problem:2:equal} satisfies
\[
\lim\limits_{t\rightarrow\infty} x(t) = 0 = \lim\limits_{t\rightarrow\infty} y(t).
\]
\end{theorem}
\section{Special case $\mu\in\mathbb{N}$ and $\nu=1$}\label{s:subset}\label{sec:multiple}
Next, we consider the case when either $\mu=k\nu$ or $\nu=k\mu$ for some $k,\mu,\nu\in\mathbb{N}$, i.e., either $\T_\mu\subset \T_\nu$ or $\T_\nu\subset \T_\mu$. Without loss of generality we only focus on the case $\mu\in\mathbb{N}$ and $\nu=1$, i.e.,
\[
\mathbb{T}_\mu = \left\lbrace 0, \mu, 2\mu, \ldots \right \rbrace = \mu \mathbb{T}_\nu ,
\]
and problem \eqref{e:problem:2} could be rewritten as (similarly as in Examples \ref{x:12} or \ref{x:dynamical:equivalence})
\begin{equation}\label{e:problem:2:mu:1}
\binom{x(t+\mu)}{y(t+\mu)}  =  \left(\begin{array}{cc}
1+\mu\alpha & \mu\beta \\
\sum_{i=0}^{\mu-1} \gamma (1+\delta)^i & (1+\delta)^\mu
\end{array}
 \right)\binom{x(t)}{y(t)} =  \Psi^{\mu,1}(\mu,0) \binom{x(t)}{y(t)},
\end{equation}

with the solution operator
\[
 \Psi^{\mu,1}(\mu,0) \coloneqq \left(\begin{array}{cc}
1+\mu\alpha & \mu\beta \\
\sum_{i=0}^{\mu-1} \gamma (1+\delta)^i & (1+\delta)^\mu
\end{array}
 \right).
\]

We have the following result regarding the stability of the origin in this case:

\begin{theorem}\label{thm:mu:1} 
Let $\lambda( \Psi^{\mu,1}(\mu,0))\subseteq B(0,1)$. Then $o$ is a globally asymptotically stable solution of the $(\mu,1)$-discrete dynamical system \eqref{e:problem:2}, i.e., every solution $(x,y)$ of \eqref{e:problem:2:mu:1} satisfies
\[
\lim\limits_{t\rightarrow\infty} x(t) = 0 = \lim\limits_{t\rightarrow\infty} y(t).
\]
\end{theorem}

\begin{proof}
We divide the proof into two parts. First, we show that 
\begin{equation}\label{e:Amo:Tx}
\lim\limits_{\substack{t\rightarrow\infty \\ t\in \mathbb{T}_\mu}} \binom{x(t)}{y(t)} = \binom{0}{0}.
\end{equation}
and next we show that this is also true for $t\in \mathbb{T}_\nu$.

The assumption $\lambda(\Psi^{\mu,1}(\mu,0))\subseteq B(0,1)$ implies that the spectral radius of $\Psi^{\mu,1}(\mu,0)$ satisfies
\[
\rho(\Psi^{\mu,1}(\mu,0)) = \max\{|\lambda_1|,|\lambda_2|\} < 1.
\]
Consequently, a standard argument (e.g., \cite[Theorem 4.4]{KP}) implies that for $\kappa$  with $\rho(\Psi^{\mu,1}(\mu,0))< \kappa <1$ there exists $K>0$ so that for each $t\in \{0,\mu, 2\mu, 3\mu, \ldots\}$ we have
\begin{equation}\label{e:Amo:Tx:estimate}
\left\lVert \binom{x(t)}{y(t)} \right\rVert = \left\lVert \binom{x(n\mu)}{y(n\mu)} \right\rVert \leq K \kappa^n \left\lVert \binom{x(0)}{y(0)} \right\rVert, \qquad (n\in\mathbb{N}).
\end{equation}
This implies that \eqref{e:Amo:Tx} holds.

Assume now that $t\notin \mathbb{T}_\mu$, i.e., $t=n\mu + m$, with $n\in\mathbb{N}$ and $m\in \{1,2,\ldots, \mu-1 \}$. Apparently, $x(n\mu)=x(n\mu+1)=x(n\mu+2)=\cdots=x(n\mu+\mu-1)$ and we have
\begin{align*}
y(n\mu+1) &= \gamma x(n\mu) + (1+\sigma) y(n\mu) \\
y(n\mu+2) &= \gamma x(n\mu+1) + (1+\sigma) y(n\mu+1) \\
	&= (\gamma + \gamma(1+\sigma)) x(n\mu) + (1+\sigma)^2 y(n\mu) \\
& \hspace*{1.2ex} \vdots \\
y(n\mu+m) &= \left(\sum_{i=0}^{m-1} \gamma (1+\delta)^i \right) x(n\mu) + (1+\sigma)^m y(n\mu).
\end{align*}
Consequently,  we can write
\[
\binom{x(t)}{y(t)} = \binom{x(n\mu+m)}{y(n\mu+m)}=
\left(\begin{array}{cc}
1 & 0 \\
\sum_{i=0}^{m-1} \gamma (1+\delta)^i & (1+\delta)^m
\end{array}
 \right)\binom{x(n\mu)}{y(n\mu)}.
\]

Then we have\footnote{For $M\in\mathbb{R}^{N\times N}$ we use the spectral matrix norm $\lVert M \rVert $
\[
\lVert M \rVert = \lVert M \rVert_{\mathop{spec}} := \sup \{\lVert Mx \rVert_2: \|x\|_2=1 \},
\]
which is equal to the largest singular value of the matrix $M$, see L\"{u}tkepohl \cite[Chapter 8]{L}.
}
\begin{align*}
\left\lVert \binom{x(t)}{y(t)} \right\rVert = \left\lVert \binom{x(n\mu+m)}{y(n\mu+m)} \right\rVert &\leq  \left\lVert \left(\begin{array}{cc}
1 & 0 \\
\sum_{i=0}^{m-1} \gamma (1+\delta)^i & (1+\delta)^m
\end{array} 
 \right) \right\rVert \cdot
 \left\lVert \binom{x(n\mu)}{y(n\mu)} \right\rVert \\
 & \leq L \left\lVert \binom{x(n\mu)}{y(n\mu)} \right\rVert,
\end{align*}
where $L$ is a constant defined by
\[
L = \max\limits_{m=1,2,\ldots, \mu-1} \left\lVert \left(\begin{array}{cc}
1 & 0 \\
\sum_{i=0}^{m-1} \gamma (1+\delta)^i & (1+\delta)^m
\end{array} 
 \right) \right\rVert.
\]

Consequently, the estimate \eqref{e:Amo:Tx:estimate} implies that
\[
\left\lVert \binom{x(t)}{y(t)} \right\rVert = \left\lVert \binom{x(n\mu+m)}{y(n\mu+m)} \right\rVert \leq 
L K \kappa^n \left\lVert \binom{x(0)}{y(0)} \right\rVert,
\]
which finishes the proof.
\end{proof}

The following counterexamples show that neither the asymptotic stability of $(\mu,\mu)$-dynamics \eqref{e:problem:2:equal} implies the asymptotic stability of $(\mu,1)$-dynamics \eqref{e:problem:2:mu:1} with the same parameters $\alpha, \beta, \gamma, \delta$, nor vice versa.

\begin{example}
If
\[
P= \left(
\begin{array}{cc}
 -\frac{1}{16} & \frac{1}{8} \\
 -\frac{1}{8} & -\frac{1}{16} \\
\end{array}
\right)
\]
and $\mu=7$, then the origin in the $(7,7)$-dynamics \eqref{e:problem:2:equal} is not asymptotically stable, since
\[
\Psi^{7,7}(7,0)= \left(
\begin{array}{cc}
 \frac{9}{16} & \frac{7}{8} \\
 -\frac{7}{8} & \frac{9}{16} \\
\end{array}
\right), 
\]
and both eigenvalues satisfy $| \lambda_{12}| \approx 1.04 >  1$. However, the origin in the $(7,1)$-dynamics \eqref{e:problem:2:mu:1} is asymptotically stable, since the eigenvalues of
\[
\Psi^{7,1}(7,0) = \left(
\begin{array}{cc}
 \frac{9}{16} & \frac{7}{8} \\
 -\frac{97\, 576\, 081}{134\, 217\, 728} & \frac{170\,859\,375}{268\,435\,456} \\
\end{array}
\right)
\]
satisfy $| \lambda_{12}| \approx 0.997 < 1$.
\end{example}

\begin{example}
If
\[
P= \left(
\begin{array}{cc}
 -\frac{1}{11} & \frac{1}{10} \\
 -\frac{2}{15} & \frac{1}{15} \\
\end{array}
\right)
\]
and $\mu=3$, then the origin in the $(3,3)$-dynamics \eqref{e:problem:2:equal} is asymptotically stable, since
\[
\Psi^{3,3}(3,0)= \left(
\begin{array}{cc}
 \frac{8}{11} & \frac{3}{10} \\
 -\frac{2}{5} & \frac{6}{5} \\
\end{array}
\right), 
\]
and both eigenvalues satisfy $| \lambda_{12}| \approx 0.996 <1$. However, the origin in the $(3,1)$-dynamics \eqref{e:problem:2:mu:1} is not asymptotically stable, since the eigenvalues of
\[
\Psi^{3,1}(3,0)= \left(
\begin{array}{cc}
 \frac{8}{11} & \frac{3}{10} \\
 -\frac{1\, 442}{3\, 375} & \frac{4\, 096}{3\,375} \\
\end{array}
\right), 
\]
satisfy $| \lambda_{12}| \approx 1.006 > 1$.
\end{example}

In the same spirit, neither the asymptotic stability of $(1,1)$-dynamics \eqref{e:problem:2:equal} implies the asymptotic stability of $(\mu,1)$-dynamics \eqref{e:problem:2:mu:1} with the same parameters $\alpha, \beta, \gamma, \delta$, nor vice versa.
\begin{example}
If
\[
P= \left(\begin{array}{cc}
 -8 & -1 \\
 \frac{1089}{40} & 4 \\
\end{array}
\right)
\]
 then the origin in the $(2,1)$-dynamics \eqref{e:problem:2:mu:1} is asymptotically stable, since the eigenvalues of
\[
\Psi^{2,1}(2,0)=\left(\begin{array}{cc}
 -17 & -2 \\
 \frac{1089}{8} & 16 \\
\end{array}
\right)
\]
are both equal to $-\frac{1}{2}$. At the same time, the origin is unstable in  the $(1,1)$-dynamics \eqref{e:problem:2:mu:1}, since one of the eigenvalues of
\[
\Psi^{1,1}(1,0)= P =\left(\begin{array}{cc}
 -8 & -1 \\
 \frac{1089}{40} & 4 \\
\end{array}
\right)
\]
is $\frac{1}{20} \left(-40-3 \sqrt{390}\right) \approx -4.96$.
\end{example}

\begin{example}
Finally, we can trivially observe that for
\[
P=\left(
\begin{array}{cc}
 -\frac{7}{4} & 0 \\
 0 & -\frac{7}{4} \\
\end{array}
\right)
\]
the origin of the $(\mu,1)$-dynamics, $\mu\in\mathbb{N}$, is stable if and only if $\mu=1$, since 
\[
\Psi^{\mu,1}(\mu,0) = \left(
\begin{array}{cc}
 1-\frac{7 \mu}{4} & 0 \\
 0 & \left(-\frac{3}{4}\right)^\mu \\
\end{array}
\right)
\]
and $1-\frac{7 \mu}{4} < -1$ for $\mu=2,3,\ldots$
\end{example}

\section{Commensurability case $\mu,\nu\in\mathbb{N}$}\label{sec:commensurable}
In this section we consider a general situation in which $\mu,\nu\in\mathbb{N}$ are not multiples of each other, i.e., $T=\mathrm{lcm}\{\mu,\nu\}>\max\{\mu,\nu\}$. We consider the following time scales
\begin{align*}
\mathbb{T}_\mu &= \left\lbrace 0, \mu, 2\mu, \ldots \right \rbrace, \\
\mathbb{T}_\nu &= \left\lbrace 0, \nu, 2\nu, \ldots \right \rbrace, \\
\mathbb{T} &= \mathbb{T}_\mu \cup \mathbb{T}_\nu, \\
\mathbb{T}_T &= \left\lbrace 0, T, 2T, \ldots \right \rbrace = \mathbb{T}_\mu \cap \mathbb{T}_\nu.
\end{align*}

For given parameters $P=\left(\begin{array}{cc}
\alpha & \beta \\
\gamma & \delta 
\end{array}
 \right)$ we can use the matrices $A_{ijkl}$ from Table \ref{t:9cases} to construct the evolution operator (see Corollaries \ref{c:matrices} and \ref{c:representation:matrix}):
\[
\Phi^{\mu,\nu} (T,0) = \prod_{m\in\T\cap[0,T)} A_{i_mj_mk_ml_m},
\] 
and the solution operator (matrix) $\Psi^{\mu,\nu} (T,0)$ defined by
\[
\Psi^{\mu,\nu} (T,0)\coloneqq
  \begin{pmatrix}
     \Phi^{\mu,\nu}_{11}(T,0) & \Phi^{\mu,\nu}_{13}(T,0)
  \\
     \Phi^{\mu,\nu}_{31}(T,0) & \Phi^{\mu,\nu}_{33}(T,0)
  \end{pmatrix}.
\]

We have the following result.

\begin{theorem}\label{thm:mu:nu}
Let $\mu, \nu\in \mathbb{N}$ and $\lambda( \Psi^{\mu,\nu} (T,0))\subseteq B(0,1)$. Then $o$ is a globally asymptotically stable solution of \eqref{e:problem:2}, i.e., every solution of \eqref{e:problem:2} satisfies
\[
\lim\limits_{t\rightarrow\infty} x(t) = 0 = \lim\limits_{t\rightarrow\infty} y(t).
\]
\end{theorem}
\begin{proof}
The proof follows the ideas of the proof of Theorem \ref{thm:mu:1}.

First, we show that
\begin{equation}\label{e:Amn:TT}
\lim\limits_{\substack{t\rightarrow\infty \\ t\in \mathbb{T}_T}} \binom{x(t)}{y(t)} = \binom{0}{0}.
\end{equation}
First, we choose $\kappa$ such that $\rho(\Psi^{\mu,\nu}(T,0))< \kappa <1$. Then, there exists $K>0$ so that for each $t\in \T_T\cap\mathbb{N}$ we have
\begin{equation}\label{e:Amn:Tx:estimate}
\left\lVert \binom{x(t)}{y(t)} \right\rVert = \left\lVert \binom{x(nT)}{y(nT)} \right\rVert \leq K \kappa^n \left\lVert \binom{x(0)}{y(0)} \right\rVert, \quad n\in\mathbb{N},
\end{equation}
which implies \eqref{e:Amn:TT}.

Next, we focus on the values of $x(t)$ and $y(t)$ on the intervals $(mT,(m+1)T)\cap\mathbb{T}$ for some $m\in\mathbb{N}$. Since the evolution operator $\Phi^{\mu,\nu}$ is defined as a product of matrices $A_{ijkl}$ from Table \ref{t:9cases} and the solution operator $\Phi^{\mu,\nu}$ as the submatrix of $\Psi^{\mu,\nu}$, we observe that with the constant
\[
L = \max\limits_{t\in\T\cap(0,T)} \| \Psi^{\mu,\nu}(t,0) \|,
\]
we have for all $t\in\T\cap(0,T)$ and $n\in\mathbb{N}$,
\[
\left\lVert \binom{x(nT+t)}{y(nT+t)} \right\rVert \leq L \left\lVert \binom{x(nT)}{y(nT)} \right\rVert.
\]
Employing the estimate \eqref{e:Amn:Tx:estimate} we get
\[
\left\lVert \binom{x(nT+t)}{y(nT+t)} \right\rVert \leq 
L K \kappa^n \left\lVert \binom{x(0)}{y(0)} \right\rVert,
\]
and the proof is complete.
\end{proof}

Before we illustrate Theorem \ref{thm:mu:nu} we introduce the notion of dynamically equivalent asynchronous dynamical systems.

\begin{definition}
Let $\mu,\nu,\hat{\mu},\hat{\nu}$ be commensurable, i.e., there exists $T=\mathrm{lcm}\lbrace \mu,\nu,\hat{\mu},\hat{\nu} \rbrace $. We say that a $(\mu,\nu)$-asynchronous discrete dynamical system with parameters $P$ and a $(\hat{\mu},\hat{\nu})$-discrete dynamical system with parameters $\hat{P}$ are \emph{dynamically equivalent} if the solution operators satisfy
\[
\Psi^{\mu,\nu}(T,0) = \hat{\Psi}^{\hat{\mu},\hat{\nu}}(T,0).
\]
\end{definition}

\begin{example}\label{x:dynamical:equivalence}
The following two asynchronous discrete dynamical systems are dynamically equivalent because in both cases they lead to the dynamics with solution operator
\[
\Psi^{\mu,\nu}(6,0)=\left(
\begin{array}{cc}
 -\frac{7}{10} & \frac{1}{10} \\
 -\frac{9}{16} & \frac{29}{80}
\end{array}
\right).
\]
Since $\lambda(\Psi^{\mu,\nu}(6,0))=\left\{\frac{1}{160} \left(-27-\sqrt{5785}\right),\frac{1}{160} \left(\sqrt{5785}-27\right)\right\}\approx\{-0.644,0.306\}$, the origin is asymptotically stable in both cases.
\begin{itemize}
\item[(a)] \textbf{$(2,3)$-asynchronous discrete dynamics}. If we consider parameters
\[
P=\begin{pmatrix}
    \alpha & \beta \\
    \gamma & \delta 
  \end{pmatrix} =  \left(
\begin{array}{cc}
 -1 & \frac{1}{5} \\
 \frac{1}{4} & -\frac{1}{4} \\
\end{array}
\right),
\]
\begin{figure}
\begin{center}
\includegraphics[width=10cm]{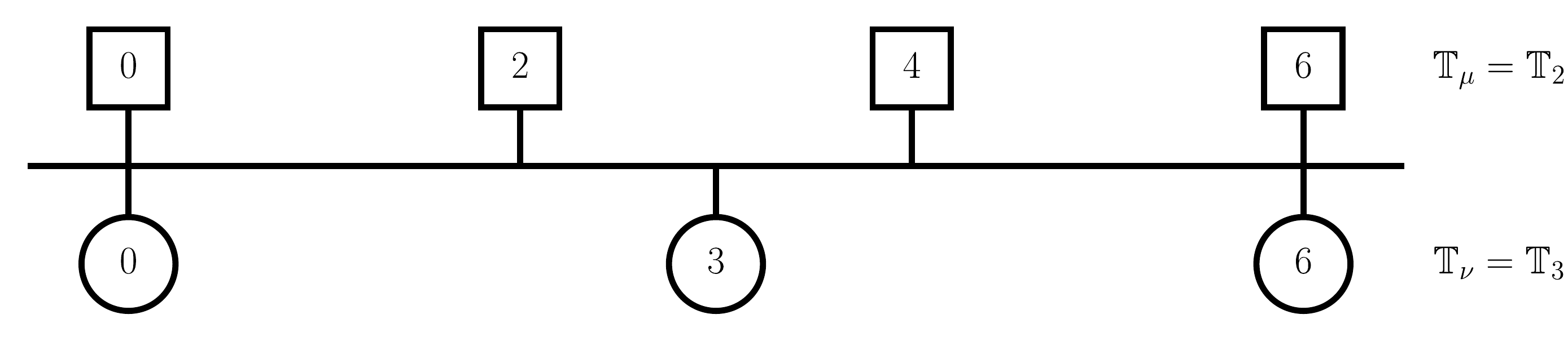}
\includegraphics[width=10cm]{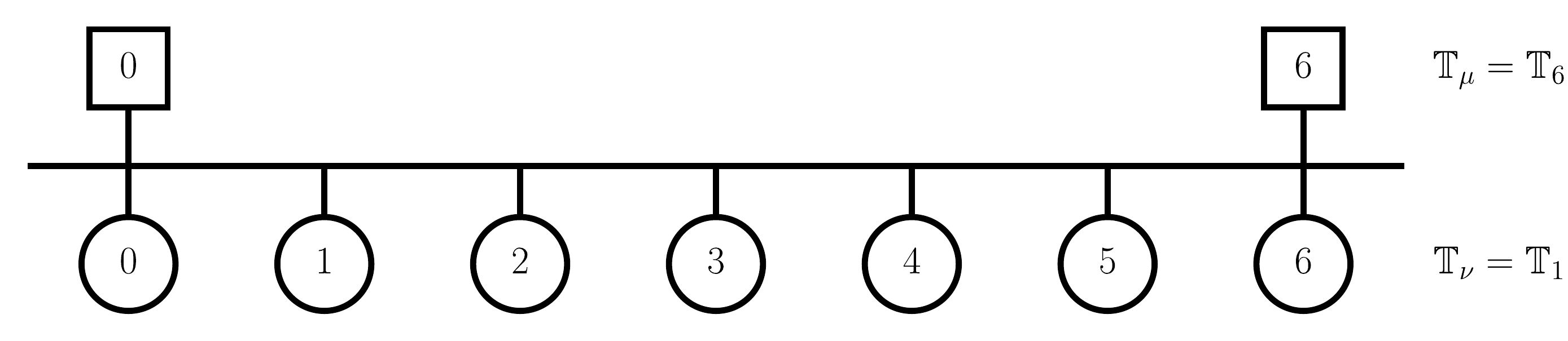}
\vspace*{-2ex}
\end{center}
\caption{Time scales related to dynamically equivalent (2,3)- and (6,1)-asynchronous discrete dynamical systems from Example \ref{x:dynamical:equivalence}.}\label{f:dynamical:equivalence:ts}
\end{figure}

we have in this case (see Figure \ref{f:dynamical:equivalence:ts} and Table \ref{t:9cases}) the evolution operator
\begin{align*}
\Phi^{2,3}(6,0)&=A_{1101}A_{0110}A_{1001}A_{1110}\\ &=\left(
\begin{array}{cccc}
 (2 \alpha +1)^3+6 \beta  \gamma  & 0 & (2 \alpha +1) (2 (2 \alpha +1) \beta +2 \beta )+2 \beta  (3 \delta +1) & 0 \\
 (2 \alpha +1)^2 & 0 & 2 (2 \alpha +1) \beta +2 \beta  & 0 \\
 3 (2 \alpha +1) \gamma +3 (3 \delta +1) \gamma  & 0 & (3 \delta +1)^2+6 \beta  \gamma  & 0 \\
 3 \gamma  & 0 & 3 \delta +1 & 0 \\
\end{array}
\right).
\end{align*}
Consequently, the solution operator is
\begin{align*}
\Psi^{2,3}(6,0)&=\left(
\begin{array}{cc}
 (2 \alpha +1)^3+6 \beta  \gamma  & (2 \alpha +1) (2 (2 \alpha +1) \beta +2 \beta )+2 \beta  (3 \delta +1) \\
 3 (2 \alpha +1) \gamma +3 (3 \delta +1) \gamma  & (3 \delta +1)^2+6 \beta  \gamma  \\
\end{array} \right) \\
&=\left(
\begin{array}{cc}
 -\frac{7}{10} & \frac{1}{10} \\
 -\frac{9}{16} & \frac{29}{80} \\
\end{array}
\right).
\end{align*}
\item[(b)] \textbf{$(6,1)$-asynchronous discrete dynamics}. If we consider parameters
\[
\displaystyle \hat{P}=\begin{pmatrix}
    \hat{\alpha} & \hat{\beta} \\
    \hat{\gamma} & \hat{\delta }
  \end{pmatrix} =  \begin{pmatrix}
 -\frac{17}{60} & \frac{1}{60} \\[1ex]
 -\frac{\frac{-9}{16}}{\sum_{i=0}^5  (\sqrt[6]{\frac{29}{60}})^i} & -1+\sqrt[6]{\frac{29}{60}} \\
\end{pmatrix}\approx \begin{pmatrix}
 -.283 & .017 \\
 -.137 & -.156 \\
\end{pmatrix},
\]
then we have (see Figure \ref{f:dynamical:equivalence:ts} and Table \ref{t:9cases})
\begin{align*}
\hat{\Phi}^{6,1}(6,0)&=\hat{A}_{0111}(\hat{A}_{0011})^4 \hat{A}_{1011}
=\begin{pmatrix}
 1+6 \hat{\alpha} & 0 & 6 \hat{\beta}  & 0 \\
 1 & 0 & 0 & 0 \\
 \hat{\gamma}\sum_{i=0}^5  (1+\hat{\delta})^i & 0 & (1+\hat{\delta} )^6 & 0 \\
  \hat{\gamma}\sum_{i=0}^4  (1+\hat{\delta})^i & 0 & (1+\hat{\delta} )^5 & 0 \\
\end{pmatrix}.
\end{align*}
Consequently, the solution operator is
\begin{align*}
\hat{\Psi}^{6,1}(6,0)&=
\begin{pmatrix}
 1+6 \hat{\alpha} &  6 \hat{\beta} \\
 \hat{\gamma}\sum_{i=0}^5  (1+\hat{\delta})^i & (1+\hat{\delta} )^6  \\
\end{pmatrix} =
\begin{pmatrix}
 -\frac{7}{10} & \frac{1}{10} \\
 -\frac{9}{16} & \frac{29}{80} \\
\end{pmatrix}.
\end{align*}
\end{itemize}
Observant readers may have noted that the parameters $\hat{P}$ can be derived backwards so that the solution matrices $\Psi^{2,3}(6,0)$ and $\hat{\Psi}^{6,1}(6,0)$ have the same form. Similarly, one could find parameters sets $\bar{P}$ for, e.g., dynamically equivalent $(1,1)$-, $(2,1)$-, $(3,2)$-, $(1,3)$-asynchronous discrete dynamical systems.
\end{example}

\begin{remark}
Note that the notion of dynamical equivalence of asynchronous discrete dynamical systems could have been alternatively introduced via the induced time-$1$ dynamics. $(\mu,\nu)$-asynchronous discrete dynamical system with parameters $P$ and a $(\hat{\mu},\hat{\nu})$-discrete dynamical system with parameters $\hat{P}$ are \emph{dynamically equivalent} if the induced time-1 operators defined by \eqref{e:timeone} are equal, i.e.,
\[
\Psi^{\mu,\nu}(1,0)=\hat{\Psi}^{\hat{\mu},\hat{\nu}}(1,0).
\]

Note that in Example \ref{x:dynamical:equivalence} both asynchronous discrete dynamical systems are associated with the complex time-1 solution operator
\[
\Psi(1,0)=\sqrt[6]{\Psi^{\mu,\nu}(6,0)}\approx \left(
\begin{array}{cc}
 0.804\, +0.492 i & 0.002\, -0.049 i \\
 -0.01+0.275 i & 0.822\, -0.027 i \\
\end{array}
\right).
\]
\end{remark}

\section{Final remarks}\label{sec:final:remarks}
Our ideas can in principle be extended to dynamical systems with more equations, e.g., 3 asynchronous discrete equations. Naturally, such a process could be computationally demanding.

However, there are two essential questions which remain open even in the case of two asynchronous equations \eqref{e:problem:2}.

First, note that the most general case we have studied was the case of commensurable $\mu,\nu$, i.e., the situation in which there exists $T=\mathrm{lcm}\{\mu,\nu\}.$ However, the cornerstone of our approach, the construction of a solution operator $\Psi(T,0)$ on the intersection time scale $\mathbb{T}_T=\T_\mu\cap \T_\nu$ cannot be used in the situation when $\mu, \nu\in\mathbb{R}^+$ are incommensurable (for example $\mu=1$ and $\nu=\pi$, etc.). In this case, there is no least common multiple $T$. The open question is how to study such $(\mu,\nu)$-asynchronous discrete dynamical systems. Under which condition is the origin of a $(\mu,\nu)$-asynchronous discrete dynamical system with incommensurable $\mu,\nu$ asymptotically stable?

The second question is motivated by counterexamples in Section \ref{s:subset} where we showed that for a given set of parameters $P$ the asymptotic stability of origin in $(\mu,1)$-asynchronous discrete dynamical systems, $\mu\in\mathbb{N}$ neither implies nor is implied by the asymptotic stability of the origin of $(\mu,\mu)$- or $(1,1)$-synchronous discrete dynamical system. Under which assumptions does the asymptotic stability of the origin in $(\mu,\nu)$-dynamics imply the asymptotic stability of the origin in $(\hat{\mu},\hat{\nu})$-dynamics?

From the point of view of applications, there are also natural questions. We can illustrate one of the key ones by our little Example \ref{x:VAR}. In the case of macroeconomic time series, can we show that in some specific instances, a variant of our asynchronous model \eqref{e:problem:2} explains the real asynchronous time series better than standard synchronous fiscal models \eqref{eq:synchronous:fiscal}? Naturally, asynchronous systems would create a realm of questions in econometrics related to the estimation of parameters, etc.

\subsubsection*{Acknowledgements}
The authors are grateful to Michal Franta, Jan Libich and Eduard Rohan for their insights from econometrics, economics and computational mechanics. PS acknowledges the support of the project LO1506 of the Czech Ministry of Education, Youth and Sports under the program NPU I.

{

}

\newpage
\begin{table}
\begin{tabularx}{\textwidth}{cccccc}
\hline\hline
{\footnotesize ${ i:=\indicator_{\Tx}(t)}$} & {\footnotesize${ j:=\indicator_{\Tx}(\sigma(t))}$} & {\footnotesize ${ k:=\indicator_{\Ty}(t)}$} & {\footnotesize ${ \ell:=\indicator_{\Ty}(\sigma(t))}$} & Pictogram & $A_{ijk\ell}$ \\ \hline
\\[-5mm]
1 & 1 & 1 & 1 &
\begin{tabular}{@{}c@{}} \\ \includegraphics[scale=.15]{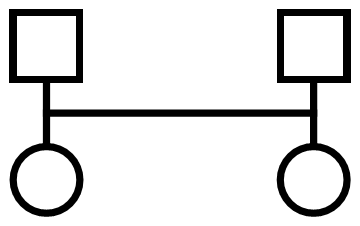}\end{tabular} & { $\begin{psmallmatrix}
 1+\mu\alpha & 0 & \mu\beta  & 0 \\
 1 & 0 & 0 & 0 \\
 \nu\gamma  & 0 & 1+\nu\delta & 0 \\
 0 & 0 & 1 & 0 \\
\end{psmallmatrix}=:A_{1111}$} \\

1 & 1 & 1 & 0 & \begin{tabular}{@{}c@{}} \\ \includegraphics[scale=.15]{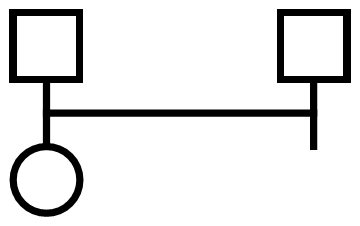}\end{tabular} &
$\begin{psmallmatrix}
 1+\mu\alpha & 0 & \mu\beta  & 0 \\
 1 & 0 & 0 & 0 \\
 0 & 0 & 1 & 0 \\
 0 & 0 & 1 & 0 \\
\end{psmallmatrix}=:A_{1110}$ \\

1 & 1 & 0 & 1 &\begin{tabular}{@{}c@{}} \\ \includegraphics[scale=.15]{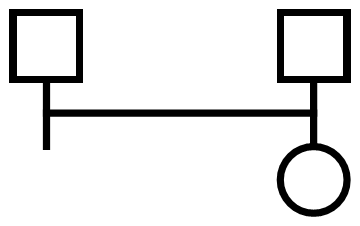}\end{tabular} &
$\begin{psmallmatrix}
 1+\mu\alpha & 0 & 0 & \mu\beta  \\
 1 & 0 & 0 & 0 \\
 0 & \nu\gamma  & 1+\nu\delta & 0 \\
 0 & 0 & 0 & 1 \\
\end{psmallmatrix}=:A_{1101}$ \\

1 & 1 & 0 &0 &\begin{tabular}{@{}c@{}} \\ \includegraphics[scale=.15]{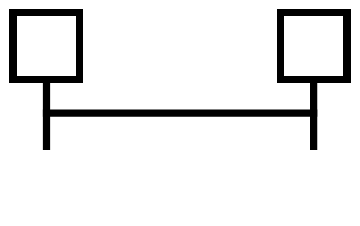}\end{tabular} &
 $\begin{psmallmatrix}
 1+\mu\alpha & 0 & 0 & \mu\beta  \\
 0 & 1 & 0 & 0 \\
 0 & 0 & 1 & 0 \\
 0 & 0 & 0 & 1 \\
\end{psmallmatrix}=:A_{1100}$ \\

1 & 0 & 1 & 1 & \begin{tabular}{@{}c@{}} \\ \includegraphics[scale=.15]{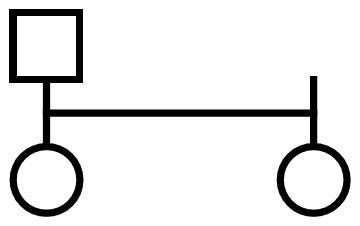}\end{tabular} &
$\begin{psmallmatrix}
 1 & 0 & 0 & 0 \\
 1 & 0 & 0 & 0 \\
 \nu\gamma  & 0 & 1+\nu\delta & 0 \\
 0 & 0 & 1 & 0 \\
\end{psmallmatrix}=:A_{1011}$ \\

1 & 0 & 0 & 1 & \begin{tabular}{@{}c@{}} \\ \includegraphics[scale=.15]{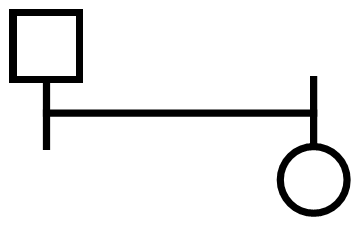}\end{tabular} &
$\begin{psmallmatrix}
 1 & 0 & 0 & 0 \\
 1 & 0 & 0 & 0 \\
 0 & \nu\gamma  & 1+\nu\delta & 0 \\
 0 & 0 & 0 & 1 \\
\end{psmallmatrix}=:A_{1001}$ \\

0 & 1 & 1 & 1 &\begin{tabular}{@{}c@{}} \\ \includegraphics[scale=.15]{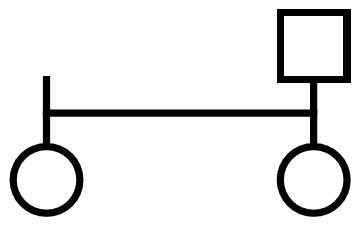}\end{tabular} &
$\begin{psmallmatrix}
 1+\mu\alpha & 0 & 0 & \mu\beta  \\
 0 & 1 & 0 & 0 \\
 0 & \nu\gamma  & 1+\nu\delta & 0 \\
 0 & 0 & 1 & 0 \\
\end{psmallmatrix}=:A_{0111}$ \\

0 & 1 & 1 & 0 &\begin{tabular}{@{}c@{}} \\ \includegraphics[scale=.15]{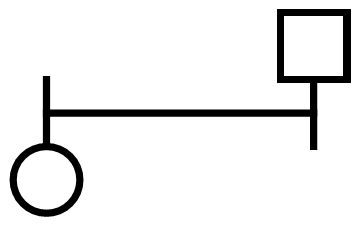}\end{tabular} &
$\begin{psmallmatrix}
 1+\mu\alpha & 0 & 0 & \mu\beta  \\
 0 & 1 & 0 & 0 \\
 0 & 0 & 1 & 0 \\
 0 & 0 & 1 & 0 \\
\end{psmallmatrix}=:A_{0110}$ \\

0 & 0 & 1 & 1 &\begin{tabular}{@{}c@{}} \\ \includegraphics[scale=.15]{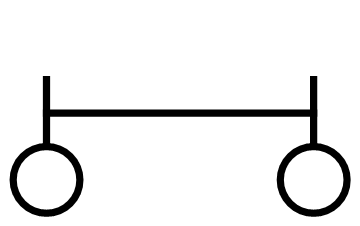}\end{tabular} &
$\begin{psmallmatrix}
 1 & 0 & 0 & 0 \\
 0 & 1 & 0 & 0 \\
 0 & \nu\gamma  & 1+\nu\delta & 0 \\
 0 & 0 & 0 & 1 \\
\end{psmallmatrix}=:A_{0011}$

\\[7ex] \hline\hline
\end{tabularx}
\caption{9 possible forms of the one-step evolution operator $A(t)$, $t,\sigma(t)\in\T$ associated with the system \eqref{(a)}, see Corollary \ref{c:representation:matrix}. The pictograms illustrate each quadruple $(i,j,k,\ell) = \big(\indicator_{\T_{\mu}}(t), \indicator_{\T_{\mu}}(\sigma(t)), \indicator_{\T_{\nu}}(t), \indicator_{\T_{\nu}}(\sigma(t))\big)$, squares correspond to $\Tx$, circles to $\Ty$, the left symbols to time $t\in\T$ and the right symbols to $\sigma(t)\in\T$. \label{t:9cases}}
\end{table}

\end{document}